\documentclass[12pt]{article}
\usepackage[a4paper,margin=1in]{geometry}
\usepackage{amssymb}
\usepackage{amsmath}
\usepackage{amsfonts}
\usepackage{amsthm}
\usepackage{color}
\usepackage{float}
\usepackage[all]{xy}
\usepackage{scalerel}
\usepackage{mathabx}
\usepackage{delimset}
\usepackage{graphicx} 
\usepackage{mathrsfs}

\newcommand{\qbinom}{\genfrac{[}{]}{0pt}{}}

\newcommand{\C}{\mathbb{C}}

\newcommand{\N}{\mathbb{N}}

\newcommand{\e}{\mathrm{e}}
\newcommand{\A}{\mathcal{A}}
\newcommand{\p}{\mathrm{P}}
\newcommand{\q}{\mathrm{Q}}
\newcommand{\AAA}{\mathrm{A}}
\newcommand{\B}{\mathrm{B}}
\newcommand{\E}{\mathrm{E}}
\newcommand{\G}{\mathrm{G}}

\newcommand{\Aa}{\mathscr{A}}
\newcommand{\rr}{\mathrm{R}}
\newcommand{\T}{\mathrm{T}}

\newtheorem{theorem}{Theorem}

\newtheorem{definition}{Definition}

\newtheorem{corollary}{Corollary}

\title{\textbf{Deformed Bivariate $q$-Appell Polynomials}}
\author{Ronald Orozco L\'opez}

\begin{document}

\maketitle

\begin{abstract}
In this paper, we introduce bivariate polynomial sets of deformed $q$-Appell type, and we study the algebraic properties of these sets. We show the relation between deformed bivariate $q$-Appell polynomials and deformed homogeneous polynomials. Next, we give some of their characterizations and algebraic structure. Then, we introduce the deformed $q$-Appell operators and obtain Mehler's and Rogers-type formulas of quasi-$q$-Appell polynomials. Finally, some examples of polynomial sequences of deformed $q$-Appell type are given: Bernoulli, Euler, and Genocchi types.
\end{abstract}
0{\bf Keywords:} $q$-Appell polynomials; deformed bivariate $q$-Appell polynomials; deformed $q$-Bernoulli polynomials, deformed $q$-Euler polynomials; deformed $q$-Genocchi polynomials.

\section{Introduction}
For every $n\geq0$, we define the $q$-numbers by $[n]_{q}=\frac{1-q^n}{1-q}$ and the $q$-factorial by $[n]_{q}!=[1]_{1}[2]_{q}\cdots[n]_{q}$. In \cite{orozco} was introduced the deformed $q$-exponential function
\begin{equation}\label{eqn_dqexp}
    \e_{q}(z,u)=
    \begin{cases}
        \sum_{n=0}^{\infty}u^{\binom{n}{2}}\frac{z^{n}}{[n]_{q}!}&\text{ if }u\neq0;\\
        1+z&\text{ if }u=0,
    \end{cases}
\end{equation}
for all $u\in\C$. Some deformed $q$-exponential functions are:
\begin{align*}
    \e_{q}(z,1)&=e_{q}(z),\ \vert z\vert<1,\\
    \e_{q}(z,q)&=\E_{q}(z),\ z\in\C,\\
    e_{q}(z,\sqrt{q})&=\mathcal{E}_{q}(z)=\sum_{n=0}^{\infty}q^{\frac{1}{2}\binom{n}{2}}\frac{z^n}{[n]_{q}!},\ z\in\C,\\
    \e_{q}(qz,q^2)&=\mathcal{R}_{q}(z)=\sum_{n=0}^{\infty}q^{n^2}\frac{z^n}{[n]_{q}!},\ z\in\C,
\end{align*}
where the $q$-exponential functions are
\begin{equation*}
    \e_{q}(z)=\sum_{n=0}^{\infty}\frac{z^n}{[n]_{q}!},
\end{equation*}
and
\begin{equation*}
    \E_{q}(z)=\sum_{n=0}^{\infty}q^{\binom{n}{2}}\frac{z^n}{[n]_{q}!},
\end{equation*}
$\mathcal{E}_{q}(z)$ is the Exton $q$-exponential function and  $\mathcal{R}_{q}(z)$ is the Rogers-Ramanujan function. The $q$-differential operator $D_{q}$ is defined by:
\begin{equation*}
    D_{q}f(x)=\frac{f(x)-f(qx)}{(1-q)x}
\end{equation*}
and the Leibniz rule for $D_{q}$
\begin{equation}\label{eqn_leibniz}
    D_{q}^{n}\{f(x)g(x)\}=\sum_{k=0}^{n}q^{k(k-n)}\qbinom{n}{k}_{q}D_{q}^{k}\{f(x)\}D_{q}^{n-k}\{g(q^{k}x)\}.
\end{equation}
Then
\begin{equation*}
    D_{q}^nx^{k}=\frac{(q;q)_{k}}{(q;q)_{k-n}}x^{k-n}.
\end{equation*}
The well-known Appell polynomials \cite{appell} $\p_{n}(x)$ are given by
\begin{align*}
    \A(t)e^{xt}=\sum_{n=0}^{\infty}\p_{n}(x)\frac{t^n}{n!},
\end{align*}
where $\A(t)$ is the determining function of the Appell polynomials. The Appell polynomials $\p_{n}(x)$ holds
\begin{equation*}
    \frac{d}{dx}\p_{n}(x)=n\p_{n-1}(x)
\end{equation*}
for $n=0,1,2,\ldots$. Many generalizations of Appell polynomials have been given: $q$-Appell polynomials of type I (Al-Salam \cite{al-salam}) given by
\begin{align}
    \A(t)\e_{q}(xt)&=\sum_{n=0}^{\infty}\p_{n,q}(x)\frac{t^n}{[n]_{q}!},\label{eqn_gf1}\\
    D_{q}\p_{n,q}(x)&=[n]_{q}\p_{n-1,q}(x),\ \ \text{for } n=0,1,2,\ldots.\label{eqn_p1}
\end{align}
The $q$-Appell polynomials of type II (Sadjang \cite{sadjang1}) given by
\begin{align}
    \A(t)\E_{q}(xt)&=\sum_{n=0}^{\infty}\p_{n,q}(x;q)\frac{t^n}{[n]_{q}!},\label{eqn_gf2}\\
    D_{q}\p_{n,q}(x)&=[n]_{q}\p_{n-1,q}(qx;q),\ \ \text{for } n=0,1,2,\ldots.\label{eqn_p2}
\end{align}
The $(p,q)$-Appell polynomials, (Sadjang \cite{sadjang2}), given by
\begin{align}
    \A(t)\e_{p,q}(xt)&=\sum_{n=0}^{\infty}\p_{n,p,q}(x;p)\frac{t^n}{[n]_{p,q}!},\label{eqn_gf3}\\
    D_{p,q}\p_{n,p,q}(x)&=[n]_{p,q}\p_{n-1,p,q}(px;p),\ \ \text{for } n=0,1,2,\ldots,\label{eqn_p3}
\end{align}
where the $(p,q)$-numbers are defined by $[n]_{p,q}=p^{\binom{n}{2}}[n]_{Q}$, $Q=q/p$, and the $(p,q)$-derivative defined as
\begin{equation*}
    D_{p,q}f(x)=\frac{f(px)-f(qx)}{(p-q)x}.
\end{equation*}
If in Eqs. (\ref{eqn_gf1}), (\ref{eqn_gf2}), and (\ref{eqn_gf3}) we use the deformed $q$-exponential function Eq.(\ref{eqn_dqexp}), then obtain the deformed $q$-Appell polynomials 
\begin{equation}\label{eqn_dappell}
    \A(t)\e_{q}(xt,u)=\sum_{n=0}^{\infty}\p_{n,q}(x;u)\frac{t^n}{[n]_{q}!}.
\end{equation}
The deformed $q$-Appell polynomials holds
\begin{equation*}
    D_{q}\p_{n,q}(x;u)=[n]_{q}\p_{n-1,q}(ux;u).
\end{equation*}
Therefore, Eq.(\ref{eqn_dappell}) is not only a generalization of the previous Appell polynomials, but it also allows us to introduce Exton and Ramanujan type Appell polynomials, and also any other family of these polynomials by simply varying the parameter $u$.

In this paper, we introduce the deformed bivariate $q$-Appell polynomials $\p_{n,q}^{(\alpha)}(x,y;u)$ of order $\alpha$, and we study the algebraic properties of these polynomials. We show the relation between deformed bivariate $q$-Appell polynomials and deformed homogeneous polynomials $\rr_{n}(x,y;u|q)$ \cite{orozco}. Next, we give some of their characterizations and algebraic structure. Then, we introduce the deformed $q$-Appell operators and obtain Mehler's and Rogers-type formulas of quasi-$q$-Appell polynomials. Finally, some examples of polynomial sequences of deformed $q$-Appell type are given: Bernoulli, Euler, and Genocchi types.

In our work, we will use the identities for binomial coefficients:
\begin{align*}
    \binom{n+k}{2}&=\binom{n}{2}+\binom{k}{2}+nk,\\
    \binom{n-k}{2}&=\binom{n}{2}+\binom{k}{2}+k(1-n).
\end{align*}
The $q$-shifted factorial is defined by
\begin{align*}
    (a;q)_{n}&=\begin{cases}
        1&\text{ if }n=0;\\
    \prod_{k=0}^{n-1}(1-q^{k}a),&\text{ if }n\neq0,\\
    \end{cases}\hspace{1cm}q\in\C.
\end{align*}
The $q$-binomial coefficient is defined by
\begin{equation*}
\qbinom{n}{k}_{q}=\frac{(q;q)_{n}}{(q;q)_{k}(q;q)_{n-k}}.
\end{equation*}

\section{Deformed bivariate $q$-Appell polynomials}

\subsection{Definition and properties}

\begin{definition}
Let $\alpha$ be an arbitrary complex number. The $u$-deformed $q$-Appell polynomials of order $\alpha$ are defined by
\begin{equation}\label{eqn_def_uqApp1}
    (\A_{q}(t))^{\alpha}\e_{q}(tx,u)=\sum_{n=0}^{\infty}\p_{n,q}^{(\alpha)}(x;u)\frac{t^{n}}{[n]_{q}!},
\end{equation}
where the determining function is 
\begin{equation}
    (\A_{q}(t))^{\alpha}=\sum_{n=0}^{\infty}a_{n}^{(\alpha)}\frac{t^n}{[n]_{q}!},
\end{equation}
with $a_{0}^{(\alpha)}\neq0,\ \A_{q}(0)\neq0.$
\end{definition}

\begin{theorem}\label{theo_iden1}
    \begin{align}
        \p_{n,q}^{(\alpha)}(x;u)=\sum_{k=0}^{n}\qbinom{n}{k}_{q}u^{\binom{n-k}{2}}a_{k}^{(\alpha)}x^{n-k}.
    \end{align}
\end{theorem}
\begin{proof}
    \begin{align*}
        (\A_{q}(t))^{\alpha}\e_{q}(tx,u)&=\left(\sum_{n=0}^{\infty}a_{n}^{(\alpha)}\frac{t^n}{[n]_{q}!}\right)\left(\sum_{n=0}^{\infty}u^{\binom{n}{2}}\frac{x^nt^n}{[n]_{q}!}\right)\\
        &=\sum_{n=0}^{\infty}\left(\sum_{k=0}^{n}\qbinom{n}{k}_{q}u^{\binom{n-k}{2}}a_{k}^{(\alpha)}x^{n-k}\right)\frac{t^{n}}{[n]_{q}!}
    \end{align*}
\end{proof}
At $x=0$, $\p_{n,q}^{(\alpha)}(0;u)=a_{n}^{(\alpha)}$.
\begin{definition}
Let $\alpha$ be an arbitrary complex number. The $u$-deformed bivariate $q$-Appell polynomials of order $\alpha$ are defined by
\begin{equation}\label{eqn_uqApp2}
    (\A_{q}(t))^{\alpha}\e_{q}(tx)\e_{q}(ty,u)=\sum_{n=0}^{\infty}\p_{n,q}^{(\alpha)}(x,y;u)\frac{t^{n}}{[n]_{q}!},
\end{equation}
where the determining function is 
\begin{equation}
    (\A_{q}(t))^{\alpha}=\sum_{n=0}^{\infty}a_{n}^{(\alpha)}\frac{t^n}{[n]_{q}!},
\end{equation}
with $a_{0}^{(\alpha)}\neq0,\ \A_{q}(0)\neq0.$
\end{definition}
If $\alpha=0$, $\p_{n,q}^{(0)}(x,y;u)=\rr_{n}(x,y;u|q)$ and if $\alpha=1$, then $\p_{n,q}^{(1)}(x,y;u)=\p_{n,q}(x,y;u)$, the $u$-deformed $q$-Appell polynomials. 
\begin{theorem}\label{theo_iden2}
For all $\alpha\in\C$
    \begin{align}
        \p_{n,q}^{(\alpha)}(x,y;u)&=\sum_{k=0}^{n}\qbinom{n}{k}_{q}u^{\binom{n-k}{2}}\p_{k,q}^{(\alpha)}(x)y^{n-k}.\label{eqn_iden21}\\
        \p_{n,q}^{(\alpha)}(x,y;u)&=\sum_{k=0}^{n}\qbinom{n}{k}_{q}\p_{k,q}^{(\alpha)}(y;u)x^{n-k}.\label{eqn_iden22}
    \end{align}
\end{theorem}
\begin{proof}
From Theorem \ref{theo_iden1}
    \begin{align*}
        (\A_{q}(t))^{\alpha}\e_{q}(tx)\e_{q}(ty,u)&=\left(\sum_{n=0}^{\infty}a_{n}^{(\alpha)}\frac{t^n}{[n]_{q}!}\right)\left(\sum_{n=0}^{\infty}\frac{x^nt^n}{[n]_{q}!}\right)\left(\sum_{n=0}^{\infty}u^{\binom{n}{2}}\frac{y^nt^n}{[n]_{q}!}\right)\\
        &=\left(\sum_{n=0}^{\infty}\sum_{k=0}^{n}\qbinom{n}{k}_{q}a_{k}^{(\alpha)}x^{n-k}\frac{t^n}{[n]_{q}!}\right)\left(\sum_{n=0}^{\infty}u^{\binom{n}{2}}\frac{y^nt^n}{[n]_{q}!}\right)\\
        &=\left(\sum_{n=0}^{\infty}\p_{n,q}^{(\alpha)}(x)\frac{t^n}{[n]_{q}!}\right)\left(\sum_{n=0}^{\infty}u^{\binom{n}{2}}\frac{y^nt^n}{[n]_{q}!}\right)\\
        &=\sum_{n=0}^{\infty}\sum_{k=0}^{n}\qbinom{n}{k}_{q}u^{\binom{n-k}{2}}\p_{k,q}^{(\alpha)}(x)y^{n-k}\frac{t^n}{[n]_{q}!}\\
        &=\sum_{n=0}^{\infty}\p_{n,q}^{(\alpha)}(x,y;u)\frac{t^n}{[n]_{q}!}.
    \end{align*}
Therefore, Eq.(\ref{eqn_iden21}) is proved. To proof Eq.(\ref{eqn_iden22}) we use
\begin{equation*}
    (\A_{q}(t))^{\alpha}\e_{q}(tx)\e_{q}(ty,u)=(\A_{q}(t)^{\alpha}\e_{q}(ty,u))\e_{q}(tx).
\end{equation*}
\end{proof}
From above theorem, $\p_{n,q}^{(\alpha)}(0,x;u)=\p_{n,q}^{(\alpha)}(x;u)$. Next, we will express the polynomials $\p_{n,q}^{(\alpha)}(x,y;u)$ as a linear combination of the polynomials $\{\rr_{k}(x,y;u|q)\}_{k=0}^{n}$.
\begin{theorem}\label{theo_iden3}
    \begin{align}
        \p_{n,q}^{(\alpha)}(x,y;u)&=\sum_{k=0}^{n}\qbinom{n}{k}_{q}a_{k}^{(\alpha)}\rr_{n-k}(x,y;u|q).
    \end{align}
\end{theorem}
\begin{proof}
\begin{align*}
    (\A_{q}(t))^{\alpha}\e_{q}(tx)\e_{q}(ty,u)&=\left(\sum_{n=0}^{\infty}a_{n}^{(\alpha)}\frac{t^n}{[n]_{q}!}\right)\left(\sum_{n=0}^{\infty}\frac{x^nt^n}{[n]_{q}!}\right)\left(\sum_{n=0}^{\infty}u^{\binom{n}{2}}\frac{y^nt^n}{[n]_{q}!}\right)\\
    &=\left(\sum_{n=0}^{\infty}a_{n}^{(\alpha)}\frac{t^n}{[n]_{q}!}\right)\left(\sum_{n=0}^{\infty}\rr_{n}(x,y;u|q)\frac{t^n}{[n]_{q}!}\right)\\
    &=\sum_{n=0}^{\infty}\left(\sum_{k=0}^{n}\qbinom{n}{k}_{q}a_{k}^{(\alpha)}\rr_{n-k}(x,y;u|q)\right)\frac{t^n}{[n]_{q}!}.
\end{align*}    
\end{proof}

\begin{theorem}\label{theo_iden4}
    \begin{equation}
        \p_{n,q}^{(\alpha)}(x,y;u)=\T(yD_{q}|u)\{\p_{n,q}^{(\alpha)}(x)\}
    \end{equation}
where $\T(yD_{q}|u)\{x^n\}=\rr_{n}(x,y;u|q)$ (see \cite{orozco}).    
\end{theorem}
\begin{proof}
From Theorem \ref{theo_iden3}
    \begin{align*}
        \p_{n,q}^{(\alpha)}(x,y;u)&=\sum_{k=0}^{n}\qbinom{n}{k}_{q}a_{k}^{(\alpha)}\rr_{n-k}(x,y;u|q)\\
        &=\sum_{k=0}^{n}\qbinom{n}{k}_{q}a_{k}^{(\alpha)}\T(yD_{q}|u)\{x^{n-k}\}\\
        &=\T(yD_{q}|u)\left\{\sum_{k=0}^{n}\qbinom{n}{k}_{q}a_{k}^{(\alpha)}x^{n-k}\right\}\\
        &=\T(yD_{q}|u)\{\p_{n,q}^{(\alpha)}(x)\}.
    \end{align*}
\end{proof}

\begin{theorem}\label{theo_iden5}
    \begin{equation}
        \T(yD_{q}|u)\{\A_{q}^{\alpha}(t)\e_{q}(tx)\}=\A_{q}^{\alpha}(t)\e_{q}(tx)\e_{q}(ty,u).
    \end{equation}
\end{theorem}
\begin{proof}
    \begin{align*}
        \T(yD_{q}|u)\{\A_{q}^{\alpha}(t)\e_{q}(tx)\}&=\T(yD_{q}|u)\left\{\sum_{n=0}^{\infty}\p_{n,q}^{(\alpha)}(x)\frac{t^n}{[n]_{q}!}\right\}\\
        &=\sum_{n=0}^{\infty}\T(yD_{q}|u)\{\p_{n,q}^{(\alpha)}(x)\}\frac{t^n}{[n]_{q}!}\\
        &=\sum_{n=0}^{\infty}\p_{n,q}^{(\alpha)}(x,y;u)\frac{t^n}{[n]_{q}!}\\
        &=\A_{q}^{\alpha}(t)\e_{q}(tx)\e_{q}(ty,u).
    \end{align*}
\end{proof}

\begin{theorem}\label{theo_iden6}
For $n\geq0$, the $q$-derivatives of $\p_{n,q}^{(\alpha)}(x,y;u)$ are:
    \begin{align}
        D_{q,x}\p_{n,q}^{(\alpha)}(x,y;u)&=[n]_{q}\p_{n-1,q}^{(\alpha)}(x,y;u).\\
        D_{q,y}\p_{n,q}^{(\alpha)}(x,y;u)&=[n]_{q}\p_{n-1,q}^{(\alpha)}(x,uy;u).
    \end{align}
\end{theorem}
\begin{proof}
From Theorem \ref{theo_iden2}
    \begin{align*}
        D_{q,x}\p_{n,q}^{(\alpha)}(x,y;u)&=D_{q,x}\left(\sum_{k=0}^{n}\qbinom{n}{k}_{q}\p_{k,q}^{(\alpha)}(y;u)x^{n-k}\right)\\
        &=\sum_{k=0}^{n-1}\qbinom{n}{k}_{q}\p_{k,q}^{(\alpha)}(y;u)[n-k]_{q}x^{n-k-1}\\
        &=[n]_{q}\sum_{k=0}^{n-1}\qbinom{n-1}{k}_{q}\p_{k,q}^{(\alpha)}(y;u)x^{n-k-1}\\
        &=[n]_{q}\p_{n-1,q}^{(\alpha)}(x,y;u).
    \end{align*}
\end{proof}

\begin{theorem}\label{theo_iden7}
    \begin{align}
        D_{q,x}^{k}\p_{n,q}^{(\alpha)}(x,y;u)&=\frac{[n]_{q}!}{[n-k]_{q}!}\p_{n-k,q}^{(\alpha)}(x,y;u).\\
        D_{q,y}^{k}\p_{n,q}^{(\alpha)}(x,y;u)&=\frac{[n]_{q}!}{[n-k]_{q}!}u^{\binom{k}{2}}\p_{n-k,q}^{(\alpha)}(x,u^{k}y;u).
    \end{align}
\end{theorem}
\begin{proof}
The proof will be by induction on $k$. For $k=1$ we have Theorem \ref{theo_iden5}. Now, suppose the statement is true for $k$ and let us prove it for $k+1$. We have
\begin{align*}
    D_{q,x}^{k+1}\{\p_{n,q}^{(\alpha)}(x,y;u)\}&=D_{q,x}D_{q,x}^{k}\{\p_{n,q}^{(\alpha)}(x,y;u)\}\\
    &=\frac{[n]_{q}!}{[n-k]_{q}!}D_{q,x}\{\p_{n-k,q}^{(\alpha)}(x,y;u)\}\\
    &=\frac{[n]_{q}!}{[n-k]_{q}!}[n-k]_{q}\p_{n-k-1,q}^{(\alpha)}(x,y;u)\\
    &=\frac{[n]_{q}!}{[n-k-1]_{q}!}u^{\binom{k+1}{2}}\p_{n-k-1,q}(x,y;u).
\end{align*}
\end{proof}

\begin{theorem}\label{theo_iden8}
    \begin{equation}
        \p_{n,q}^{(\alpha)}(x,y;u)=\sum_{k=0}^{n}\qbinom{n}{k}_{q}u^{\binom{k}{2}}\AAA_{k,q}(a;u)y^{k}\p_{n-k,q}^{(\alpha)}(x,y;u)
    \end{equation}
    where $\AAA_{n,q}(a;u)$ is a sequence satisfying the recursion relation
    \begin{equation}
        \sum_{k=0}^{n}\qbinom{n}{k}_{q}u^{k(k-n)}a^{k}\AAA_{n-k,q}(a;u)=1.    
    \end{equation}
\end{theorem}
\begin{proof}
    \begin{align*}
        (\A_{q}(t))^{\alpha}\e_{q}(xt)\e_{q}(yt,u)=\frac{\e_{q}(yt,u)}{\e_{q}(ayt,u)}\cdot(\A_{q}(t))^{\alpha}\e_{q}(xt)\e_{q}(ayt,u).
    \end{align*}
We have the identity
\begin{equation*}
    \frac{\e_{q}(yt,u)}{\e_{q}(ayt,u)}=\sum_{n=0}^{\infty}u^{\binom{n}{2}}\AAA_{n,q}(a;u)\frac{t^ny^n}{[n]_{q}!}.
\end{equation*}
Then
\begin{align*}
    &(\A_{q}(t))^{\alpha}\e_{q}(xt)\e_{q}(yt,u)\\
    &=\left(\sum_{n=0}^{\infty}u^{\binom{n}{2}}\AAA_{n,q}(a;u)y^n\frac{t^{n}}{[n]_{q}!}\right)\left(\sum_{n=0}^{\infty}\p_{n,q}^{(\alpha)}(x,y;u)\frac{t^n}{[n]_{q}!}\right)\\
    &=\sum_{n=0}^{\infty}\left(\sum_{k=0}^{n}\qbinom{n}{k}_{q}u^{\binom{k}{2}}\AAA_{k,q}(a;u)y^{k}\p_{n-k,q}^{(\alpha)}(x,y;u)\right)\frac{t^{n}}{[n]_{q}!}
\end{align*}
\end{proof}
The first few values of sequence $\AAA_{n,q}(a;u)$ are:
\begin{align*}
    \AAA_{0,q}(a;u)&=1,\\
    \AAA_{1,q}(a;u)&=1-a,\\
    \AAA_{2,q}(a;u)&=1-[2]_{q}u^{-1}(1-a)a-a^2,\\
    \AAA_{3,q}(a;u)&=1-[3]_{q}u^{-2}a-q[3]_{q}u^{-2}a^{2}(1-a)-[3]_{q}u^{-2}a^{3},\\
\end{align*}
If $u=1$, then $\AAA_{n,q}(a;1)=(a;q)_{n}$ and if $u=q$, then $\AAA_{n,q}(a;q)=(a;q^{-1})_{n}$.

\begin{theorem}\label{theo_iden9}
Let $n\in\N$ and $\alpha,\beta$ be real or complex numbers. Then we have
    \begin{equation}
        \p_{n,q}^{(\alpha+\beta)}(x,y;u)=\sum_{k=0}^{n}\qbinom{n}{k}_{q}\p_{k,q}^{(\alpha)}(x)\p_{n-k,q}^{(\beta)}(y;v).
    \end{equation}
\end{theorem}
\begin{proof}
On the one side
    \begin{align*}
        (\A_{q}(t))^{\alpha+\beta}\e_{q}(tx)\e_{q}(ty,u)&=\left(\sum_{n=0}^{\infty}\p_{n,q}^{(\alpha)}(x)\frac{t^n}{[n]_{q}!}\right)\left(\sum_{n=0}^{\infty}\p_{n,q}^{(\beta)}(y;v)\frac{t^{n}}{[n]_{q}!}\right)\\
        &=\sum_{n=0}^{\infty}\left(\sum_{k=0}^{n}\qbinom{n}{k}_{q}\p_{k,q}^{(\alpha)}(x)\p_{n-k,q}^{(\beta)}(y,u)\right)\frac{t^n}{[n]_{q}!}.
    \end{align*}
On the other hand
\begin{align*}
    (\A_{q}(t))^{\alpha+\beta}\e_{q}(tx)\e_{q}(yt,u)=\sum_{n=0}^{\infty}\p_{n,q}^{(\alpha+\beta)}(x,y;u)\frac{t^{n}}{[n]_{q}!}.
\end{align*}
\end{proof}

\begin{corollary}
Let $n\in\N$ and $\alpha$ be real or complex numbers. Then we have
    \begin{equation}
        \p_{n,q}^{(2\alpha)}(x,y;u)=\sum_{k=0}^{n}\qbinom{n}{k}_{q}\p_{k,q}^{(\alpha)}(x)\p_{n-k,q}^{(\alpha)}(y;u).
    \end{equation}
\end{corollary}

\begin{corollary}
Let $n\in\N$ and $\alpha$ be real or complex numbers. Then we have
    \begin{equation}
        \rr_{n}(x,y;u|q)=\sum_{k=0}^{n}\qbinom{n}{k}_{q}\p_{k,q}^{(\alpha)}(x)\p_{n-k,q}^{(-\alpha)}(y;u).
    \end{equation}
\end{corollary}

\subsection{Characterizations}

\begin{theorem}\label{theo-charac}
Let $\{\p_{n,q}^{(\alpha)}(x;u)\}_{n}^{\infty}$ is a sequence of polynomials. Then the following are equivalent:
\begin{enumerate}
    \item $\{\p_{n,q}(x;u)\}_{n=0}^{\infty}$ is a sequence of deformed $q$-Appell polynomials.
    \item There exists a sequence $(a_{k}^{(\alpha)})_{k\geq0}$, independent of $n$, with $a_{0}^{(\alpha)}\neq0$ and such that
    \begin{equation*}
        \p_{n,q}^{(\alpha)}(x;u)=\sum_{k=0}^{n}\qbinom{n}{k}_{q}u^{\binom{n-k}{2}}a_{k}^{(\alpha)}x^{n-k}.
    \end{equation*}
    \item $\{\p_{n,q}(x;u)\}_{n=0}^{\infty}$ can be defined by means of following generating function 
    \begin{equation*}
        (\A_{q}(t))^{\alpha}\e_{q}(tx)=\sum_{n=0}^{\infty}\p_{n,q}^{(\alpha)}(x;u)\frac{t^n}{[n]_{q}!},
    \end{equation*}
    where 
    \begin{equation*}
        \A_{q}^{\alpha}(t)=\sum_{n=0}^{\infty}a_{n}^{(\alpha)}\frac{t^n}{[n]_{q}!},\ a_{0}^{(\alpha)}\neq0,\ \A_{q}(0)\neq0.
    \end{equation*}
    \item There exists a sequence $(a_{k}^{(\alpha)})_{k\geq0}$, independent of $n$ with $a_{0}\neq0$ and such that
    \begin{equation*}
        \p_{n,q}(x;u)=\left(\sum_{k=0}^{\infty}u^{\binom{n-k}{2}}a_{k}^{(\alpha)}\frac{D_{q}^{k}}{[k]_{q}!}\right)x^n.
    \end{equation*}
\end{enumerate}
\end{theorem}
\begin{proof}
(1)$\Rightarrow$(2). Suppose that $\p_{n}(x;u)$ is a $u$-deformed $q$-Appell polynomial such that
\begin{equation}\label{eqn_pn}
    \p_{n}(x;u)=\sum_{k=0}^{n}\qbinom{n}{k}_{q}a^{(\alpha)}_{n,k}u^{\binom{n-k}{2}}x^{n-k},\ \ n=1,2,3,\ldots,
\end{equation}
where the coefficients $a^{(\alpha)}_{n,k}$ depend on $n$ and $k$ and $a_{n,0}^{(\alpha)}\neq0$. By applying the operator $D_{q}$ to each member of Eq.(\ref{eqn_pn}) we have
\begin{equation}\label{eqn-der-pn}
    \p_{n-1}(ux;u)=\sum_{k=0}^{n-1}\qbinom{n-1}{k}_{q}a^{(\alpha)}_{n,k}u^{\binom{n-1-k}{2}}(ux)^{n-k-1},\ \ n=1,2,3,\ldots,
\end{equation}
Shifting the index $n\rightarrow n+1$ in Eq.(\ref{eqn-der-pn}) and making the substitution $x\rightarrow u^{-1}x$, we get
\begin{equation}\label{eqn_pn2}
    \p_{n}(x;u)=\sum_{k=0}^{n}\qbinom{n}{k}_{q}a_{n+1,k}u^{\binom{n-k}{2}}x^{n-k},\ \ n=1,2,3,\ldots,
\end{equation}
Comparing Eq.(\ref{eqn_pn}) and Eq.(\ref{eqn_pn2}), we have $a_{n+1,k}=a_{n,k}$ for all $k$ and $n$, and therefore $a_{n+1,k}=a_{k}$ is independent of $n$. \\
(2)$\Rightarrow$(3). From (2) we have
\begin{align*}
    \sum_{n=0}^{\infty}\p_{n,q}(x;u)\frac{t^n}{[n]_{q}!}&=\sum_{n=0}^{\infty}\left(\sum_{k=0}^{n}\qbinom{n}{k}_{q}u^{\binom{n-k}{2}}a_{k}^{(\alpha)}x^{n-k}\right)\frac{t^n}{[n]_{q}!}\\
    &=\left(\sum_{n=0}^{\infty}a_{n}^{(\alpha)}\frac{t^n}{[n]_{q}!}\right)\left(\sum_{n=0}^{\infty}u^{\binom{n}{2}}\frac{(xt)^n}{[n]_{q}!}\right)\\
    &=\A_{q}(t)\e_{q}(xt,u).
\end{align*}
(3)$\Rightarrow$(1). Assume that $\{\p_{n,q}(x;u)\}$ is generated by
\begin{equation*}
    \A_{q}(t)\e_{q}(xt,u)=\sum_{n=0}^{\infty}\p_{n,q}(x;u)\frac{t^n}{[n]_{q}!}.
\end{equation*}
On the one side, applying the operator $D_{q,x}$, with respect to the variable $x$, to each side of this equation, we get
\begin{equation*}
    t\A_{q}(t)\e_{q}(uxt,u)=\sum_{n=0}^{\infty}D_{q,x}\p_{n,q}(x;u)\frac{t^{n}}{[n]_{q}!}.
\end{equation*}
On the other hand, 
\begin{align*}
    t\A_{q}(t)\e_{q}(uxt,u)&=\sum_{n=0}^{\infty}\p_{n,q}(ux;u)\frac{t^{n+1}}{[n]_{q}!}\\
    &=\sum_{n=0}^{\infty}[n]_{q}\p_{n-1,q}(ux;u)\frac{t^{n}}{[n]_{q}!}.
\end{align*}
By comparing the coefficients of $t^n$, we obtain (1). (2)$\Longleftrightarrow$(4) its obvious since $D_{q}^{k}x^{n}=0$ for $k>n$. This ends the proof of the theorem.
\end{proof}

\subsection{Algebraic structure}

Let $\{f_{n}(x)\}_{n=0}^{\infty}$ be a given polynomial set, and we denote this by a single symbol $f$ and refer to $f_{n}(x)$ as the $n$-th component of $f$. As was done in [?,?], we define on the set $\mathcal{P}$ of all polynomial sequences the following three operations $+,\cdot$ and $*$. The first one is given by the rule that $f+g$ is the polynomial sequence whose $n$-th component is $f_{n}(x)+g_{n}(x)$ provided that the degree of $f_{n}(x)+g_{n}(x)$ is exactly $n$. On the other hand, if $f$ and $g$ are the sets whose $n$-th components are, restively,
\begin{equation*}
    f_{n}(x)=\sum_{k=0}^{n}f(n,k)x^{k},\ \ \ g_{n}(x)=\sum_{k=0}^{n}g(n,k)x^k,
\end{equation*}
then $f*g$ is the polynomial set whose $n$-th component is
\begin{equation*}
    (f*g)_{n}(x)=\sum_{k=0}^{n}f(n,k)g_{k}(x).
\end{equation*}
If $\alpha$ is a real or complex number, then $\alpha f$ is the polynomial set whose $n$-th component is $\alpha f_{n}(x)$. We obviously have
\begin{align*}
    f+g&=g+f\ \text{ for all }f,g\in\mathcal{P},\\
    (\alpha f*g)&=(f*\alpha g)=\alpha(f*g).
\end{align*}
We denote the class of all $u$-deformed $q$-Appell sets by $\mathfrak{A}(q;u)$. In $\mathfrak{A}(q;u)$ the identity element with respect $*$ is the $u$-deformed $q$-Appell sets $I=\{x^n\}$. Note that $I$ has the determining function $\A_{q}(t,u)=1$. We have the following theorem.
\begin{theorem}
Let $f,g,h\in\mathfrak{A}(q;u)$ with the determining functions $\A_{q}(t)$, $\mathcal{B}_{q}(t)$ and $\mathcal{C}_{q}(t)$, respectively. Then 
\begin{enumerate}
    \item $f+g\in\mathfrak{A}(q;u)$ if $\A_{q}(0)+\mathcal{B}_{q}(0)\neq0$.
    \item $f+g$ belongs to the determining function $\A_{q}(t)+\mathcal{B}_{q}(t)$.
    \item $f+(g+h)=(f+g)+h$. 
\end{enumerate}
\end{theorem}
The proof of the following result is given.
\begin{theorem}\label{theo_prod_qappel}
If $f,g,h\in\mathfrak{A}(q;u)$ with determining functions $\A_{q}(t)$, $\mathcal{B}_{q}(t)$ and $\mathcal{C}_{q}(t)$, respectively, then 
\begin{enumerate}
    \item $f*g\in\mathfrak{A}(q;u)$.
    \item $f*g=g*f$.
    \item $f*g$ belongs to determining function $\A_{q}(t)\mathcal{B}_{q}(t)$.
    \item $f*(g*h)=(f*g)*h$.
\end{enumerate}
\end{theorem}
\begin{proof}
It is enough to prove the first part of the theorem. The rest follows directly. From Theorem \ref{theo-charac}, we may put
\begin{equation*}
    \p_{n,q}^{(\alpha)}(x;u)=\sum_{k=0}^{n}\qbinom{n}{k}_{q}u^{\binom{k}{2}}a_{k}^{(\alpha)}x^{n-k}=\sum_{k=0}^{n}\qbinom{n}{k}_{q}u^{\binom{n-k}{2}}a_{n-k}^{(\alpha)}x^k
\end{equation*}
so that 
\begin{equation*}
    \A_{q}^{\alpha}(t)=\sum_{n=0}^{\infty}a_{n}^{(\alpha)}\frac{t^n}{[n]_{q}!}.
\end{equation*}
Hence
\begin{align*}
    \sum_{n=0}^{\infty}(f*g)_{n}(x;u)\frac{t^n}{[n]_{q}!}&=\sum_{n=0}^{\infty}\left(\sum_{k=0}^{n}\qbinom{n}{k}_{q}a_{n-k}^{(\alpha)}\mathrm{Q}_{k,q}^{(\beta)}(x;u)\right)\frac{t^n}{[n]_{q}!}\\
    &=\left(\sum_{n=0}^{\infty}a^{(\alpha)}_{n}\frac{t^n}{[n]_{q}!}\right)\left(\sum_{n=0}^{\infty}\mathrm{Q}^{(\beta)}_{n,q}(x;u)\frac{t^n}{[n]_{q}!}\right)\\
    &=\A_{q}^{\alpha}(t)\mathcal{B}^{\beta}_{q}(t)\e_{q}(xt,u).
\end{align*}
This ends the proof of the theorem.
\end{proof}
\begin{corollary}
Let $f\in\mathfrak{A}(q;u)$, then $f$ has an inverse with respect to $*$, i.e. there is a set $g\in\mathfrak{A}(q;u)$ such that
\begin{equation*}
    f*g=I.
\end{equation*}
\end{corollary}
Indeed, $g$ belongs to the determining function $\A_{q}(t)^{-1}$ where $\A_{q}(t)$ is the determining function of $f$. We shall denote $g$ by $f^{-1}$. Theorem \ref{theo_prod_qappel} and its corollary allow us to define $f^0=I$, $f^n=f*f^{n-1}$, where $n$ is a non-negative, and $f^{-n}=f^{-1}*f^{-n+1}$. Therefore, with the above, we have proven that the system $\mathfrak{A}(q;u)$ is a commutative group. In particular, this leads to the fact that if 
\begin{equation*}
    f*g=h
\end{equation*}
and if any two of the elements $f,g,h$ are $u$-deformed $q$-Appell then then third is also $q$-Appell. 

\section{Deformed $q$-Appell operators}

\begin{definition}
We define the following $q$-Appell operators: the deformed $q$-Appell operator
    \begin{align}
        \Aa_{\alpha}(yD_{q}|u)&=\sum_{k=0}^{\infty}u^{\binom{k}{2}}a_{k}^{(\alpha)}\frac{y^{k}}{[k]_{q}!}D_{q}^{k}.
    \end{align}
and the deformed bivariate $q$-Appell operator
\begin{equation}
    \Aa_{\alpha}(x,y;D_{q}|u)=\sum_{k=0}^{\infty}\p_{k,q}^{(\alpha)}(x;u)\frac{y^{k}}{[k]_{q}!}D_{q}^{k}.
\end{equation}
\end{definition}

\begin{definition}
The $u$-deformed homogeneous quasi-$q$-Appell polynomials of order $\alpha$ are defined by
    \begin{equation}\label{eqn_dhq_appell}
        \q_{n,q}^{(\alpha)}(x,y;u)=\sum_{k=0}^{n}\qbinom{n}{k}_{q}u^{\binom{k}{2}}a_{k}^{(\alpha)}y^{k}x^{n-k}.
    \end{equation}
The $u$-deformed trivariate quasi-$q$-Appell polynomials of order $\alpha$ are defined by
    \begin{equation}\label{eqn_dtq_appell}
        \q_{n,q}^{(\alpha)}(x,y,z;u)=\sum_{k=0}^{n}\qbinom{n}{k}_{q}\p_{k,q}^{(\alpha)}(x;u)y^{k}z^{n-k}.
    \end{equation}
\end{definition}
The $q$-derivatives of polynomials in Eqs. (\ref{eqn_dhq_appell}) and \ref{eqn_dtq_appell} are, respectively
\begin{align*}
    D_{q,x}\q_{n,q}^{(\alpha)}(x,y;u)&=[n]_{q}\q_{n-1,q}^{(\alpha)}(x,y;u),\\
    D_{q,y}\q_{n,q}^{(\alpha)}(x,y;u)&=[n]_{q}\sum_{k=0}^{n-1}\qbinom{n-1}{k}_{q}u^{\binom{k}{2}}a_{k+1}^{(\alpha)}(uy)^{k}x^{n-1-k}
\end{align*}
and
\begin{align*}
    D_{q,x}\q_{n,q}^{(\alpha)}(x,y,z;u)&=[n]_{q}y\q_{n-1,q}^{(\alpha)}(x,y,z;u),\\
    D_{q,y}\q_{n,q}^{(\alpha)}(x,y,z;u)&=[n]_{q}\sum_{k=0}^{n-1}\qbinom{n-1}{k}_{q}\p_{k+1,q}^{(\alpha)}(x;u)y^{k}z^{n-1-k},\\
    D_{q,z}\q_{n,q}^{(\alpha)}(x,y,z;u)&=[n]_{q}\q_{n-1,q}^{(\alpha)}(x,y,z;u).
\end{align*}
Then, the polynomials $\q_{n,q}^{(\alpha)}(x,y;u)$ and $\q_{n,q}^{(\alpha)}(x,y,z;u)$ are not deformed $q$-Appell polynomials. The quasi-$q$-Appell and $q$-Appell polynomials are relating in the following way
\begin{align}
    \p_{n,q}^{(\alpha)}(x;u)&=u^{\binom{n}{2}}\q_{n,q}^{(\alpha)}(x,u^{1-2n};u),\\
    \p_{n,q}^{(\alpha)}(x,y;u)&=\q_{n,q}^{(\alpha)}(x,1,y;u).
\end{align}

\begin{theorem}
    \begin{equation}
        \q_{n,q}^{(\alpha)}(x,y,z;u)=\Aa_{\alpha}(x,y,D_{q}|u)\{z^n\}.
    \end{equation}
\end{theorem}
\begin{proof}
    \begin{align*}
        \Aa_{\alpha}(x,y,D_{q}|u)\{z^{n}\}&=\sum_{k=0}^{\infty}\p_{k,q}^{(\alpha)}(x;u)\frac{y^{k}}{[k]_{q}!}D_{q}^{k}\{z^n\}\\
        &=\sum_{k=0}^{n}\qbinom{n}{k}_{q}\p_{k,q}^{(\alpha)}(x;u)y^kz^{n-k}\\
        &=\q_{n,q}^{(\alpha)}(x,y,z;u).
    \end{align*}
\end{proof}

\begin{theorem}
    \begin{equation}
        \sum_{n=0}^{\infty}\q_{n,q}^{(\alpha)}(x,y,z;u)\frac{t^{n}}{[n]_{q}!}=\e_{q}(zt)\A_{q}^{\alpha}(yt)\e_{q}(xyts,u).
    \end{equation}
\end{theorem}
\begin{proof}
   \begin{align*}
       \sum_{n=0}^{\infty}\q_{n,q}^{(\alpha)}(x,y,z;u)\frac{t^{n}}{[n]_{q}!}&=\sum_{n=0}^{\infty}\Aa_{\alpha}(x,y,D_{q,z}|u)\{z^{n}\}\frac{t^{n}}{[n]_{q}!}\\
       &=\Aa_{\alpha}(x,y,D_{q,z}|u)\left\{\sum_{n=0}^{\infty}\frac{(zt)^n}{[n]_{q}!}\right\}\\
       &=\Aa_{\alpha}(x,y,D_{q,z}|u)\left\{\e_{q}(zt)\right\}\\
       &=\sum_{k=0}^{\infty}\frac{\p_{k,q}^{(\alpha)}(x;u)y^{k}}{[k]_{q}!}D_{q,z}^{k}\{\e_{q}(zt)\}\\
       &=\e_{q}(zt)\sum_{k=0}^{\infty}\p_{k,q}^{(\alpha)}(x;u)\frac{(ty)^{k}}{[k]_{q}!}\\
       &=\e_{q}(zt)\A_{q}^{\alpha}(yt)\e_{q}(xyts,u).
   \end{align*} 
\end{proof}

\begin{theorem}
    \begin{equation}
        \sum_{n=0}^{\infty}q^{\binom{n}{2}}\q_{n,q}^{(\alpha)}(x,y,z;u)\frac{t^{n}}{[n]_{q}!}=\E_{q}(zt)\sum_{k=0}^{\infty}q^{\binom{k}{2}}\frac{\p_{k,q}^{(\alpha)}(x;u)}{(-(1-q)zt;q)_{k}[k]_{q}!}(ty)^{k}.
    \end{equation}
\end{theorem}
\begin{proof}
   \begin{align*}
       \sum_{n=0}^{\infty}q^{\binom{n}{2}}\q_{n,q}^{(\alpha)}(x,y,z;u)\frac{t^{n}}{[n]_{q}!}&=\sum_{n=0}^{\infty}q^{\binom{n}{2}}\Aa_{\alpha}(x,yD_{q}|u)\{z^{n}\}\frac{t^{n}}{[n]_{q}!}\\
       &=\Aa_{\alpha}(x,y,D_{q}|u)\left\{\sum_{n=0}^{\infty}q^{\binom{n}{2}}\frac{(zt)^n}{[n]_{q}!}\right\}\\
       &=\Aa_{\alpha}(x,y,D_{q}|u)\left\{\E_{q}(zt)\right\}\\
       &=\sum_{k=0}^{\infty}\p_{k,q}^{(\alpha)}(x;u)\frac{y^{k}}{[k]_{q}!}D_{q}^{k}\{\E_{q}(zt)\}\\
       &=\sum_{k=0}^{\infty}q^{\binom{k}{2}}\p_{k,q}^{(\alpha)}(x;u)\frac{(ty)^{k}}{[k]_{q}!}\E_{q}(q^kzt)\\
       &=\E_{q}(zt)\sum_{k=0}^{\infty}q^{\binom{k}{2}}\frac{\p_{k,q}^{(\alpha)}(x;u)}{(-(1-q)zt;q)_{k}[k]_{q}!}(ty)^{k}.
   \end{align*} 
\end{proof}

\begin{theorem}[{\bf Mehler's formula}]
    \begin{align}
        &\sum_{n=0}^{\infty}\q_{n,q}^{(\alpha)}(x,y,z;u)\p_{n,q}^{(\beta)}(w)\frac{t^n}{[n]_{q}!}\nonumber\\
        &=\e_{q}(wzt)\sum_{i=0}^{\infty}\frac{\A_{q,i}^{\alpha}(ywt)(yt)^i}{[i]_{q}!}\sum_{k=0}^{\infty}u^{\binom{k}{2}}\frac{\A_{q,k+i}^{\beta}(zt)((1-q)wzt;q)_{k+i}(xyt)^k}{[k]_{q}!}\e_{q}(q^{i}u^{k}ywzt,u),
    \end{align}
where
\begin{equation}
    \A_{q,k}^{\beta}(t)=\sum_{n=0}^{\infty}a_{n+k}^{(\alpha)}\frac{t^n}{[n]_{q}!}.
\end{equation}
\end{theorem}
\begin{proof}
    \begin{align*}
        &\sum_{n=0}^{\infty}\q_{n,q}^{(\alpha)}(x,y,z;u)\p_{n,q}^{(\beta)}(w)\frac{t^n}{[n]_{q}!}\\
        &=\sum_{n=0}^{\infty}\Aa_{\alpha}(x,y,D_{q,z}|u)\{z^n\}\p_{n,q}^{(\beta)}(w)\frac{t^n}{[n]_{q}!}\\
        &=\Aa_{\alpha}(x,y,D_{q,z}|u)\left\{\sum_{n=0}^{\infty}\p_{n,q}^{(\beta)}(w)\frac{(zt)^n}{[n]_{q}!}\right\}\\
        &=\Aa_{\alpha}(x,y,D_{q,z}|u)\left\{\A_{q}^{\beta}(zt)\e_{q}(wzt)\right\}\\
        &=\sum_{n=0}^{\infty}\frac{\p_{n,q}^{(\alpha)}(x;u)y^{n}}{[n]_{q}!}D_{q,z}^{n}\left\{\A_{q}^{\beta}(zt)\e_{q}(wzt)\right\}\\
        &=\sum_{n=0}^{\infty}\frac{\p_{n,q}^{(\alpha)}(x;u)y^{n}}{[n]_{q}!}\sum_{k=0}^{n}q^{k(k-n)}\qbinom{n}{k}_{q}D_{q,z}^{k}\{\A_{q}^{\beta}(zt)\}D_{q,z}^{n-k}\{\e_{q}(q^{k}wzt)\}\\
        &=\sum_{n=0}^{\infty}\frac{\p_{n,q}^{(\alpha)}(x;u)y^{n}}{[n]_{q}!}\sum_{k=0}^{n}\qbinom{n}{k}_{q}t^k\A_{q,k}^{\beta}(zt)(wt)^{n-k}\e_{q}(q^{k}wzt)\\
        &=\e_{q}(wzt)\sum_{n=0}^{\infty}\frac{\p_{n,q}^{(\alpha)}(x;u)(yt)^{n}}{[n]_{q}!}\sum_{k=0}^{n}\qbinom{n}{k}_{q}\A_{q,k}^{\beta}(zt)w^{n-k}((1-q)wzt;q)_{k}\\
        &=\e_{q}(wzt)\sum_{k=0}^{\infty}\frac{\A_{q,k}^{\beta}(zt)((1-q)wzt;q)_{k}(yt)^k}{[k]_{q}!}\sum_{n=0}^{\infty}\frac{\p_{n+k,q}^{(\alpha)}(x;u)}{[n]_{q}!}(ytw)^{n}.
    \end{align*}
As
\begin{align*}
    \sum_{n=0}^{\infty}\frac{(ywt)^{n}}{[n]_{q}!}\p_{n+k,q}^{(\alpha)}(x;u)&=\frac{1}{(yw)^k}D_{q,t}^k\left\{\A_{q}^{\alpha}(ywt)\e_{q}(ywxt,u)\right\}\\
    &=\sum_{i=0}^{k}\qbinom{k}{i}_{q}u^{\binom{k-i}{2}}x^{k-i}\A_{q,i}^{\alpha}(ywt)\e_{q}(q^{i}u^{k-i}ywxt,u),
\end{align*}
then
\begin{align*}
    &\sum_{n=0}^{\infty}\q_{n,q}^{(\alpha)}(x,y,z;u)\p_{n,q}^{(\beta)}(w)\frac{t^n}{[n]_{q}!}\\
    &=\e_{q}(wzt)\sum_{k=0}^{\infty}\frac{\A_{q,k}^{\beta}(zt)((1-q)wzt;q)_{k}(yt)^k}{[k]_{q}!}\\
    &\hspace{2cm}\times\sum_{i=0}^{k}\qbinom{k}{i}_{q}u^{\binom{k-i}{2}}x^{k-i}\A_{q,i}^{\alpha}(ywt)\e_{q}(q^{i}u^{k-i}ywxt,u)\\
    &=\e_{q}(wzt)\sum_{i=0}^{\infty}\frac{\A_{q,i}^{\alpha}(ywt)(yt)^i}{[i]_{q}!}\sum_{k=0}^{\infty}u^{\binom{k}{2}}\frac{\A_{q,k+i}^{\beta}(zt)((1-q)wzt;q)_{k+i}(xyt)^k}{[k]_{q}!}\e_{q}(q^{i}u^{k}ywzt,u).
\end{align*}

\end{proof}

\begin{theorem}[{\bf Rogers formula}]
    \begin{align}
        &\sum_{n=0}^{\infty}\sum_{m=0}^{\infty}\q_{n+m,q}^{(\alpha)}(x,y,z|u)\frac{t^{n}}{[n]_{q}!}\frac{s^{m}}{[m]_{q}!}\nonumber\\
        &\hspace{1cm}=\e_{q}(zt)\e_{q}(zs)\sum_{n=0}^{\infty}\frac{\p_{n,q}^{(\alpha)}(x;u)y^n}{[n]_{q}!}\sum_{k=0}^{n}\qbinom{n}{k}_{q}t^{k}s^{n-k}((1-q)zs;q)_{k}.
    \end{align}
\end{theorem}
\begin{proof}
    \begin{align*}
        &\sum_{n=0}^{\infty}\sum_{m=0}^{\infty}\q_{n+m,q}^{(\alpha)}(x,y,z|u)\frac{t^{n}}{[n]_{q}!}\frac{s^{m}}{[m]_{q}!}\\
        &=\sum_{n=0}^{\infty}\sum_{m=0}^{\infty}\Aa_{\alpha}(x,y,D_{q,z}|u)\{z^{n+m}\}\frac{t^{n}}{[n]_{q}!}\frac{s^{m}}{[m]_{q}!}\\
        &=\Aa_{\alpha}(x,y,D_{q,z}|u)\left\{\sum_{n=0}^{\infty}\sum_{m=0}^{\infty}\frac{(zt)^{n}}{[n]_{q}!}\frac{(zs)^{m}}{[m]_{q}!}\right\}\\
        &=\Aa_{\alpha}(x,y,D_{q,z}|u)\left\{\e_{q}(zt)\e_{q}(zs)\right\}\\
        &=\sum_{n=0}^{\infty}\frac{\p_{n,q}^{(\alpha)}(x;u)y^n}{[n]_{q}!}D_{q,z}^{n}\left\{\e_{q}(zt)\e_{q}(zs)\right\}\\
        &=\sum_{n=0}^{\infty}\frac{\p_{n,q}^{(\alpha)}(x;u)y^n}{[n]_{q}!}\sum_{k=0}^{n}q^{k(k-n)}\qbinom{n}{k}_{q}D_{q,z}^{k}\left\{\e_{q}(zt)\right\}D_{q,z}^{n-k}\{\e_{q}(q^kzs)\}\\
        &=\sum_{n=0}^{\infty}\frac{\p_{n,q}^{(\alpha)}(x;u)y^n}{[n]_{q}!}\sum_{k=0}^{n}\qbinom{n}{k}_{q}t^{k}\e_{q}(zt)s^{n-k}\e_{q}(q^kzs)\\
        &=\e_{q}(zt)\e_{q}(zs)\sum_{n=0}^{\infty}\frac{\p_{n,q}^{(\alpha)}(x;u)y^n}{[n]_{q}!}\sum_{k=0}^{n}\qbinom{n}{k}_{q}t^{k}s^{n-k}((1-q)zs;q)_{k}.
    \end{align*}
\end{proof}

\section{Examples of deformed $q$-Appell polynomials}

\subsection{$u$-deformed $q$-Bernoulli numbers and polynomials}

The $u$-deformed $q$-Bernoulli polynomials of order $\alpha$ are defined by
\begin{equation}
    \left(\frac{t}{\e_{q}(t)-1}\right)^{\alpha}\e_{q}(tx,u)=\sum_{n=0}^{\infty}\B_{n,q}^{(\alpha)}(x;u)\frac{t^{n}}{[n]_{q}!},
\end{equation}
where
\begin{align}
        \B_{n,q}^{(\alpha)}(x;u)=\sum_{k=0}^{n}\qbinom{n}{k}_{q}u^{\binom{n-k}{2}}\B_{k,q}^{(\alpha)}x^{n-k}
\end{align}
and the $\B_{n,q}^{(\alpha)}$ are the $q$-Bernoulli numbers of order $\alpha$, defined by the following generating function
\begin{equation}\label{eqn_qBerNum}
        \left(\frac{t}{\e_{q}(t)-1}\right)^{\alpha}=\sum_{n=0}^{\infty}\B_{n,q}^{(\alpha)}\frac{t^n}{[n]_{q}!}.
    \end{equation}    
The $u$-deformed bivariate $q$-Bernoulli polynomials of order $\alpha$ are defined by
\begin{equation}
    \left(\frac{t}{\e_{q}(t)-1}\right)^{\alpha}\e_{q}(tx)\e_{q}(ty,u)=\sum_{n=0}^{\infty}\B_{n,q}^{(\alpha)}(x,y;u)\frac{t^{n}}{[n]_{q}!},
\end{equation}
Some properties of the $\B_{n,q}^{(\alpha)}(x,y;q)$ are:

For all $\alpha\in\C$ and for $n\geq0$,
    \begin{align}
        \B_{n,q}^{(\alpha)}(x,y;u)&=\sum_{k=0}^{n}\qbinom{n}{k}_{q}u^{\binom{n-k}{2}}\B_{k,q}^{(\alpha)}(x)y^{n-k}.\\
        \B_{n,q}^{(\alpha)}(x,y;u)&=\sum_{k=0}^{n}\qbinom{n}{k}_{q}\B_{k,q}^{(\alpha)}(y;u)x^{n-k}.\\
        \B_{n,q}^{(\alpha)}(x,y;u)&=\sum_{k=0}^{n}\qbinom{n}{k}_{q}u^{\binom{k}{2}}\AAA_{k,q}(a;u)y^{k}\B_{n-k,q}^{(\alpha)}(x,y;u).
    \end{align}
Its relation with the deformed homogeneous polynomials and with the deformed $q$-exponential operator
\begin{align}
        \B_{n,q}^{(\alpha)}(x,y;u)&=\sum_{k=0}^{n}\qbinom{n}{k}_{q}\B_{k,q}^{(\alpha)}\rr_{n-k}(x,y;u|q),\\
        \rr_{n}(x,y;u|q)&=\sum_{k=0}^{n}\qbinom{n}{k}_{q}\B_{k,q}^{(\alpha)}(x)\B_{n-k,q}^{(-\alpha)}(y;u).
    \end{align}
and
\begin{align}
        \B_{n,q}^{(\alpha)}(x,y;u)&=\T(yD_{q}|u)\{\B_{n,q}^{(\alpha)}(x)\},\\
        \T(yD_{q}|u)\left\{\left(\frac{t}{\e_{q}(t)-1}\right)^{\alpha}\e_{q}(tx)\right\}&=\left(\frac{t}{\e_{q}(t)-1}\right)^{\alpha}\e_{q}(tx)\e_{q}(ty,u).
    \end{align}
Its $q$-derivatives,
\begin{align}
        D_{q,x}\B_{n,q}^{(\alpha)}(x,y;u)&=[n]_{q}\B_{n-1,q}^{(\alpha)}(x,y;u).\\
        D_{q,y}\B_{n,q}^{(\alpha)}(x,y;u)&=[n]_{q}\B_{n-1,q}^{(\alpha)}(x,uy;u).
    \end{align}
Addiction properties:
Let $n\in\N$ and $\alpha,\beta$ be real or complex numbers. Then we have
    \begin{align}
        \B_{n,q}^{(\alpha+\beta)}(x,y;u)&=\sum_{k=0}^{n}\qbinom{n}{k}_{q}\B_{k,q}^{(\alpha)}(x)\B_{n-k,q}^{(\beta)}(y;v).\\
        \B_{n,q}^{(2\alpha)}(x,y;u)&=\sum_{k=0}^{n}\qbinom{n}{k}_{q}\B_{k,q}^{(\alpha)}(x)\B_{n-k,q}^{(\alpha)}(y;u).
    \end{align}

\subsection{$u$-deformed $q$-Euler numbers and polynomials}

The $u$-deformed $q$-Euler polynomials of order $\alpha$ are defined by
\begin{equation}
    \left(\frac{2}{\e_{q}(t)+1}\right)^{\alpha}\e_{q}(tx,u)=\sum_{n=0}^{\infty}\E_{n,q}^{(\alpha)}(x;u)\frac{t^{n}}{[n]_{q}!},
\end{equation}
where
\begin{align}
        \E_{n,q}^{(\alpha)}(x;u)=\sum_{k=0}^{n}\qbinom{n}{k}_{q}u^{\binom{n-k}{2}}\E_{k,q}^{(\alpha)}x^{n-k}
\end{align}
and the $\E_{n,q}^{(\alpha)}$ are the $q$-Euler numbers of order $\alpha$, defined by the following generating function
\begin{equation}\label{eqn_qEulerNum}
        \left(\frac{2}{\e_{q}(t)+1}\right)^{\alpha}=\sum_{n=0}^{\infty}\E_{n,q}^{(\alpha)}\frac{t^n}{[n]_{q}!}.
    \end{equation}    
The $u$-deformed bivariate $q$-Euler polynomials of order $\alpha$ are defined by
\begin{equation}
    \left(\frac{2}{\e_{q}(t)+1}\right)^{\alpha}\e_{q}(tx)\e_{q}(ty,u)=\sum_{n=0}^{\infty}\E_{n,q}^{(\alpha)}(x,y;u)\frac{t^{n}}{[n]_{q}!},
\end{equation}
Some properties of the $\E_{n,q}^{(\alpha)}(x,y;q)$ are:

For all $\alpha\in\C$ and for $n\geq0$,
    \begin{align}
        \E_{n,q}^{(\alpha)}(x,y;u)&=\sum_{k=0}^{n}\qbinom{n}{k}_{q}u^{\binom{n-k}{2}}\E_{k,q}^{(\alpha)}(x)y^{n-k}.\\
        \E_{n,q}^{(\alpha)}(x,y;u)&=\sum_{k=0}^{n}\qbinom{n}{k}_{q}\E_{k,q}^{(\alpha)}(y;u)x^{n-k}.\\
        \E_{n,q}^{(\alpha)}(x,y;u)&=\sum_{k=0}^{n}\qbinom{n}{k}_{q}u^{\binom{k}{2}}\AAA_{k,q}(a;u)y^{k}\E_{n-k,q}^{(\alpha)}(x,y;u).
    \end{align}
Its relation with the deformed homogeneous polynomials and with the deformed $q$-exponential operator
\begin{align}
        \E_{n,q}^{(\alpha)}(x,y;u)&=\sum_{k=0}^{n}\qbinom{n}{k}_{q}\E_{k,q}^{(\alpha)}\rr_{n-k}(x,y;u|q),\\
        \rr_{n}(x,y;u|q)&=\sum_{k=0}^{n}\qbinom{n}{k}_{q}\E_{k,q}^{(\alpha)}(x)\E_{n-k,q}^{(-\alpha)}(y;u).
    \end{align}
and
\begin{align}
        \E_{n,q}^{(\alpha)}(x,y;u)&=\T(yD_{q}|u)\{\E_{n,q}^{(\alpha)}(x)\},\\
        \T(yD_{q}|u)\left\{\left(\frac{2}{\e_{q}(t)+1}\right)^{\alpha}\e_{q}(tx)\right\}&=\left(\frac{2}{\e_{q}(t)+1}\right)^{\alpha}\e_{q}(tx)\e_{q}(ty,u).
    \end{align}
Its $q$-derivatives,
\begin{align}
        D_{q,x}\E_{n,q}^{(\alpha)}(x,y;u)&=[n]_{q}\E_{n-1,q}^{(\alpha)}(x,y;u).\\
        D_{q,y}\E_{n,q}^{(\alpha)}(x,y;u)&=[n]_{q}\E_{n-1,q}^{(\alpha)}(x,uy;u).
    \end{align}
Addiction properties:
Let $n\in\N$ and $\alpha,\beta$ be real or complex numbers. Then we have
    \begin{align}
        \E_{n,q}^{(\alpha+\beta)}(x,y;u)&=\sum_{k=0}^{n}\qbinom{n}{k}_{q}\E_{k,q}^{(\alpha)}(x)\E_{n-k,q}^{(\beta)}(y;v).\\
        \E_{n,q}^{(2\alpha)}(x,y;u)&=\sum_{k=0}^{n}\qbinom{n}{k}_{q}\E_{k,q}^{(\alpha)}(x)\E_{n-k,q}^{(\alpha)}(y;u).
    \end{align}

\subsection{$u$-deformed $q$-Genocchi numbers and polynomials}

The $u$-deformed $q$-Genocchi polynomials of order $\alpha$ are defined by
\begin{equation}
    \left(\frac{2t}{\e_{q}(t)+1}\right)^{\alpha}\e_{q}(tx,u)=\sum_{n=0}^{\infty}\G_{n,q}^{(\alpha)}(x;u)\frac{t^{n}}{[n]_{q}!},
\end{equation}
where
\begin{align}
        \G_{n,q}^{(\alpha)}(x;u)=\sum_{k=0}^{n}\qbinom{n}{k}_{q}u^{\binom{n-k}{2}}\G_{k,q}^{(\alpha)}x^{n-k}
\end{align}
and the $\G_{n,q}^{(\alpha)}$ are the $q$-Genocchi numbers of order $\alpha$, defined by the following generating function
\begin{equation}\label{eqn_qGenNum}
        \left(\frac{2t}{\e_{q}(t)+1}\right)^{\alpha}=\sum_{n=0}^{\infty}\G_{n,q}^{(\alpha)}\frac{t^n}{[n]_{q}!}.
    \end{equation}    
The $u$-deformed bivariate $q$-Genocchi polynomials of order $\alpha$ are defined by
\begin{equation}
    \left(\frac{2t}{\e_{q}(t)+1}\right)^{\alpha}\e_{q}(tx)\e_{q}(ty,u)=\sum_{n=0}^{\infty}\G_{n,q}^{(\alpha)}(x,y;u)\frac{t^{n}}{[n]_{q}!},
\end{equation}
Some properties of the $\G_{n,q}^{(\alpha)}(x,y;q)$ are:

For all $\alpha\in\C$ and for $n\geq0$,
    \begin{align}
        \G_{n,q}^{(\alpha)}(x,y;u)&=\sum_{k=0}^{n}\qbinom{n}{k}_{q}u^{\binom{n-k}{2}}\G_{k,q}^{(\alpha)}(x)y^{n-k}.\\
        \G_{n,q}^{(\alpha)}(x,y;u)&=\sum_{k=0}^{n}\qbinom{n}{k}_{q}\G_{k,q}^{(\alpha)}(y;u)x^{n-k}.\\
        \G_{n,q}^{(\alpha)}(x,y;u)&=\sum_{k=0}^{n}\qbinom{n}{k}_{q}u^{\binom{k}{2}}\AAA_{k,q}(a;u)y^{k}\G_{n-k,q}^{(\alpha)}(x,y;u).
    \end{align}
Its relation with the deformed homogeneous polynomials and with the deformed $q$-exponential operator is
\begin{align}
        \G_{n,q}^{(\alpha)}(x,y;u)&=\sum_{k=0}^{n}\qbinom{n}{k}_{q}\G_{k,q}^{(\alpha)}\rr_{n-k}(x,y;u|q),\\
        \rr_{n}(x,y;u|q)&=\sum_{k=0}^{n}\qbinom{n}{k}_{q}\G_{k,q}^{(\alpha)}(x)\G_{n-k,q}^{(-\alpha)}(y;u).
    \end{align}
and
\begin{align}
        \G_{n,q}^{(\alpha)}(x,y;u)&=\T(yD_{q}|u)\{\G_{n,q}^{(\alpha)}(x)\},\\
        \T(yD_{q}|u)\left\{\left(\frac{2t}{\e_{q}(t)+1}\right)^{\alpha}\e_{q}(tx)\right\}&=\left(\frac{2t}{\e_{q}(t)+1}\right)^{\alpha}\e_{q}(tx)\e_{q}(ty,u).
    \end{align}
Its $q$-derivatives,
\begin{align}
        D_{q,x}\G_{n,q}^{(\alpha)}(x,y;u)&=[n]_{q}\G_{n-1,q}^{(\alpha)}(x,y;u).\\
        D_{q,y}\G_{n,q}^{(\alpha)}(x,y;u)&=[n]_{q}\G_{n-1,q}^{(\alpha)}(x,uy;u).
    \end{align}
Addiction properties:
Let $n\in\N$ and $\alpha,\beta$ be real or complex numbers. Then we have
    \begin{align}
        \G_{n,q}^{(\alpha+\beta)}(x,y;u)&=\sum_{k=0}^{n}\qbinom{n}{k}_{q}\G_{k,q}^{(\alpha)}(x)\G_{n-k,q}^{(\beta)}(y;v).\\
        \G_{n,q}^{(2\alpha)}(x,y;u)&=\sum_{k=0}^{n}\qbinom{n}{k}_{q}\G_{k,q}^{(\alpha)}(x)\G_{n-k,q}^{(\alpha)}(y;u).
    \end{align}


\begin{thebibliography}{99}

\bibitem{appell}
P. Appell, 
Sur une classe de polynomes, 
Ann. Sci. Ec. Norm. Supér. 
\textbf{9} (1880) 119--144.

\bibitem{al-salam}
W. A. Al-Salam, 
$q$-Appell polynomials, 
Ann. Mat. Pura Appl. 
\textbf{77} (1967), 31--45.

\bibitem{orozco}
R. Orozco, 
Deformed homogeneous polynomials and the deformed $q$-exponential operator, 
arXiv:2306.07431v4, 
(2024).

\bibitem{sadjang1}
P. N. Sadjang, 
On new $q$-analogue of Appell polynomials, 
arXiv:1801.08859v1, 
(2018).

\bibitem{sadjang2}
P. N. Sadjang, 
On $(p,q)$-Appell polynomials, 
Anal. Math. 
\textbf{45} (2019) 583--598.



\end{thebibliography}
\end{document}